\documentclass[12pt]{article}
\usepackage{amsmath,amssymb,amsfonts}
\usepackage{amsthm}
\usepackage{amscd, latexsym}
\usepackage{verbatim}
\usepackage{graphicx}
\usepackage{mathtools}
\usepackage{url}
\usepackage{enumerate}
\usepackage{booktabs}
\usepackage{lineno}

\newtheorem{theorem}{Theorem}

\newtheorem{lemma}[theorem]{Lemma}
\newtheorem{proposition}[theorem]{Proposition}

\theoremstyle{definition}
\newtheorem{remark}[theorem]{Remark}
\newtheorem{example}[theorem]{Example}

\newtheorem{definition}[theorem]{Definition}

\newcommand{\R}{{\mathbb R}}
\newcommand{\C}{{\mathbb C}}
\newcommand{\N}{{\mathbb N}}
\newcommand{\Z}{{\mathbb Z}}
\DeclareMathOperator{\im}{im}
\DeclareMathOperator{\rank}{rank}
\DeclareMathOperator{\diag}{diag}
 
 \newcommand{\dist}[0]{\mathrm{dist}}

\newcommand{\spt}[0]{\mathrm{spt}}

\DeclareMathAccent{\Circ}{\mathalpha}{operators}{"17}

\newcommand{\e}{{\bf e}}

\renewcommand{\i}[0]{\mathrm{i}}
\renewcommand{\Re}{\operatorname{Re}}
\renewcommand{\Im}{\mathrm{Im}}

\newcommand{\eps}{\varepsilon}

\begin{document}
\title{\vspace*{-10mm}
A remark on local activity and passivity}
\author{B.\ Garay%
\footnote{\tt garay@digitus.itk.ppke.hu\rm, Faculty of Information Technology and Bionics, P\'{a}zm\'{a}ny Catholic University Budapest, Hungary}
,
S.\ Siegmund%
\footnote{\tt stefan.siegmund@tu-dresden.de\rm, Institute for Analysis \& Center for Dynamics, Department of Mathematics, TU Dresden, Germany}
,
S.\ Trostorff%
\footnote{\tt sascha.trostorff@tu-dresden.de\rm, Institute for Analysis, Department of Mathematics, TU Dresden, Germany}
{} and
M.\ Waurick%
\footnote{\tt  m.waurick@bath.ac.uk\rm, Department of Mathematical Sciences, University of Bath, United Kingdom,
\tt marcus.waurick@tu-dresden.de\rm, Institute for Analysis, Department of Mathematics, TU Dresden, Germany}
}
\date{\today}
\maketitle
\begin{abstract}
We study local activity and its opposite, local passivity, for linear systems and show that generically an eigenvalue of the system matrix with positive real part implies local activity. If all state variables are port variables we prove that the system is locally active if and only if the system matrix is not dissipative. Local activity was suggested by Leon Chua as an indicator for the emergence of complexity of nonlinear systems. We propose an abstract scheme which indicates how local activity  could be applied to nonlinear systems and list open questions about possible consequences for complexity.
\end{abstract}
%
%
%

\section{Introduction}

There is a huge literature on notions of passivity for linear control systems
\begin{equation}\label{eq01}
   \dot x(t) = A x(t) + B u(t), 
   \qquad 
   y(t) = C x(t) + D u(t),
\end{equation}
with matrices $A \in \R^{n \times n}$, $B \in \R^{n \times m}$, $C \in \R^{k \times n}$, $D \in \R^{k \times m}$ (see e.g.\ \cite{TucsnakWeiss14}, \cite[Chapter 2]{BrogliatoEtal2007} and the references therein). The fundamental concept of passivity has been motivated by the study of linear electric circuits. It is derived from mathematical properties of a function of a complex variable, called \emph{positive real function}  (see \cite[Definition 4]{Chua05} or \cite[Definition 5.1]{Wing2008}). Its $n$-dimensional extension, called \emph{positive real impedance matrix} or \emph{positive real admittance matrix}, characterizes the impedance of arbitrary linear electrical circuits composed of resistors, inductors and capacitors. \eqref{eq01} is called \emph{impedance passive} if for every continuous input signal $u \colon \R_{\geq 0} \rightarrow \R^m$ and initial value $x_0 \in \R^n$ its differentiable solution $x \colon \R_{\geq 0} \rightarrow \R^n$ with $x(0) = x_0$ satisfies
\begin{equation*}
   \frac{d}{dt} \|x(t)\|^2 
   \leq
   2 \langle u(t), y(t) \rangle
   \quad (t \geq 0).
\end{equation*} 
The terminology dates back to Otto  Brune's PhD dissertation \cite{Brune1931a, Brune1931b} and the classic linear circuit theory of yore. Impedance passivity of \eqref{eq01} is equivalent to (see also 
Staffans \cite{Staffans02, Staffans03} for an infinite-dimensional analogue) 
\begin{equation}\label{eq02}
   \langle A x_0 + B u_0, x_0 \rangle 
   \leq
   \langle u_0, C x_0 + D u_0 \rangle
   \quad (x_0 \in \R^n, u_0 \in \R^m).
\end{equation}
Chua \cite[Definition 2]{Chua05} calls a linear system
\begin{equation}\label{eq03}
   \dot x(t) = A x(t) + P u(t)
\end{equation}
with $A \in \R^{n \times n}$ and a diagonal projection matrix $P \in \R^{n \times n}$, \emph{locally passive} if for all $T>0$ and continuous input signals $u \colon [0,T] \rightarrow \R^n$ the inequality $\int_0^T \langle x(t), Pu(t) \rangle dt \ge 0$ holds for the solution $x$ of \eqref{eq03} with $x(0)=0$. Using \eqref{eq02}, an extension of system \eqref{eq03} 
\begin{equation}\label{eq03ext}
   \dot x(t) = A x(t) + P u(t),
   \qquad
   y(t) = P x(t),
\end{equation}
is impedance passive if and only if $\langle A x_0 + P u_0, x_0 \rangle \leq \langle u_0, P x_0 \rangle$ which is equivalent to
\begin{equation}\label{eq04}
   \langle A x_0 , x_0 \rangle 
   \leq 0
   \quad (x_0 \in \R^n).
\end{equation}
If $P$ is the identity matrix, we show in Theorem \ref{thm:passivity} that condition \eqref{eq04} is also equivalent to system \eqref{eq03} being locally passive, i.e.\ impedance passivity for system \eqref{eq03ext} and local passivity for system \eqref{eq03} agree if $P = I$. However, to the best of our knowledge, if $P$ is not the identity then a characterization of local passivity via properties of $A$ and $P$ is still open. Nevertheless, we show in Proposition \ref{prop:dissipative}, in particular for diagonal projection matrices $P$, that the condition $\langle P A x_0, x_0 \rangle \leq 0$ for all $x_0 \in \R^n$ implies that system \eqref{eq03} is locally passive.

System \eqref{eq03} is called \emph{locally active} \cite[Definition 1]{Chua05}, if it is not locally passive. We show in Theorem \ref{thm:generic-local-activity} that generically instability of $A$ implies local activity of \eqref{eq03}. More precisely, we construct an open and dense subset of $\R^{n \times n}$ such that if a matrix $A$ in that set has an eigenvalue with positive real part then \eqref{eq03} is locally active. 

Local activity was introduced by Leon Chua for nonlinear systems to shed light on the emergence of complexity (see also \cite{MainzerChua13} and the many references therein). We agree with Chua who writes in \cite{Chua05} that ``many scientists have struggled to uncover the elusive origin of complexity, and its various conceptually equivalent jargons, such as emergence, self-organization, synergetics, nonequilibrium phase transitions, cooperative phenomena, etc.''. This makes it all the more gratifying if a testable mathematical condition like local activity \cite{Chua98, Chua05, DogaruChua98} is suggested as a concept which indicates the emergence of complexity. Local activity as a mathematical concept was proposed in \cite{Chua98} with a circuit-theoretic perspective and studied for discrete reaction-diffusion equations \cite{Chua05} and many other systems \cite{MainzerChua13}.

In this paper we take a mathematical perspective on the concept of local activity as defined in \cite[Definition 1]{Chua05} and discuss some of its properties in Section \ref{sec:local-activity}. In Section \ref{sec:nonlinear} we formulate an abstract scheme which indicates how the local activity for linear systems could be applied to nonlinear systems like the FitzHugh--Nagumo equation with dissipation \cite{ItohChua07} and a discrete reaction-diffusion equation \cite[Equation (1)]{Chua05} at a single equilibrium. We propose open questions on local activity and its consequences for complexity. Appendix \ref{sec:appendix} contains a genericity result, as well as a short proof in the Hilbert space setting of a characterization of local passivity by dissipativity in case the projection $P$ is the identity.

For an interval $I \subseteq \R$ and a Hilbert space $H$, we denote by $C(I, H)$ and $L^2(I,H)$ the spaces of continuous functions and square-integrable functions $f \colon I \rightarrow H$, respectively. $C_{c}^{\infty}(\mathbb{R})$ is the space of infinitely often differentiable functions $f \colon \R \rightarrow \R$ with compact support $\operatorname{spt}(f) = \overline{\{x \in \R \colon f(x) \neq 0\}}$. $\e_1, \dots, \e_n$ is the standard basis of $\R^n$ and for a matrix $A \in \R^{n \times n}$ its set of eigenvalues is denoted by $\sigma(A)$. For a set $D$, we denote its characteristic function by $\chi_D$.

\section{Local activity}\label{sec:local-activity}

Local activity is a concept for linear differential equations, which arise e.g.\ as linearizations of nonlinear differential equations at equilibria. A nonlinear differential equation with several equilibria is called locally active if there exists at least one equilibrium whose associated equation \eqref{eq:linear-P} is locally active (see Definition \ref{def:local-activity} below). For a discussion on nonlinear differential equations, see Section \ref{sec:nonlinear}.

\begin{definition}[Local activity with port variables, see {\cite[Definition 1]{Chua05}}]\label{def:local-activity}
Let $A \in \R^{n \times n}$ and $P \in \R^{n \times n}$ be a projection. Consider the following class of differential equations
\begin{equation}\label{eq:linear-P}
  \dot x(t) = A x(t) + P u(t)
\end{equation}
with $u \in C(\R_{\geq 0}, \R^n)$. Equation \eqref{eq:linear-P}, or equivalently the pair $(A,P)$, is called 
\emph{locally active} if there exist $T>0$ and $u\in C([0,T],\R^n)$ such that the solution $x\in C([0,T],\R^n)$ of the initial value problem \eqref{eq:linear-P}, $x(0) = 0$, satisfies
\[
   W_T(u) \coloneqq \int_0^T \langle x(t),Pu(t)\rangle dt < 0.
\]
Equation \eqref{eq:linear-P} is called \emph{locally passive} if it is not locally active, i.e.\ if for all $T>0$ and $u\in C([0,T],\R^n)$ the inequality $W_T(u) \geq 0$ holds.

If $\e_1, \dots, \e_m \in \im P$ (in particular if $P = \diag(1,\dots,1,0,\dots,0)$ with $\rank P = m$) then the first $m$ variables $x_1, \dots, x_m$ of \eqref{eq:linear-P} are called \emph{port variables}. 
\end{definition}

Chua \cite[Theorem 4]{Chua05} states a characterization of local activity for projections of the form $P = \diag(1,\dots,1,0,\dots,0)$ in terms of four properties of an appropriate Laplace transform, the so-called complexity function \cite[Formula (18)]{Chua05} (see also \cite[Theorem 2.14]{BrogliatoEtal2007}). The following theorem characterizes local activity by properties of the system matrix $A$ in case the projection is trivial.

\begin{theorem}[Local activity if all variables are port variables]\label{thm:local-activity-allport}
Consider equation \eqref{eq:linear-P} with $P = I$. Then the following three statements are equivalent:
\begin{enumerate}[(i)]
  \item \eqref{eq:linear-P} is locally active,
  \item $A$ is not dissipative, i.e.\ there exists $x \in \R^n$ with $\langle Ax,x\rangle > 0$,
  \item there exists a positive eigenvalue of $A + A^\top$.
\end{enumerate}
\end{theorem}

\begin{proof}
$(i) \Leftrightarrow (ii)$. This is the statement of Theorem \ref{thm:passivity} in Appendix \ref{sec:appendix} applied to the Hilbert space $H = \R^n$. 

$(ii) \Rightarrow (iii)$. Assume that $(iii)$ does not hold. There exists an orthogonal basis of $\R^n$ of eigenvectors $v_1, \dots, v_n$ of $A + A^\top$, i.e.\ $(A + A^\top)v_i = \lambda_i v_i$ with $\lambda_i \leq 0$. We show the opposite of $(ii)$, i.e.\ $\langle Ax,x\rangle \leq 0$ for all $x \in \R^n$. To this end, let $x \in \R^n$. Then $x = \sum_{i=1}^{n} \mu_i v_i$ with $\mu_i \in \R$ and $\langle Ax,x\rangle = \frac{1}{2} \langle (A + A^\top)x,x\rangle = \frac{1}{2}
\langle \sum_{i=1}^{n} \mu_i (A + A^\top) v_i, \sum_{i=1}^{n} \mu_i v_i\rangle = \frac{1}{2} \sum_{i=1}^{n} \lambda_i \mu_i^2 \|v_i\|^2 \leq 0$.

$(iii) \Rightarrow (ii)$. If $(A + A^\top) v = \lambda v$ with $v \in \R^n \setminus \{0\}$ and $\lambda \in \R_{>0}$, then
$\langle A v, v \rangle
   =
   \frac{1}{2} \langle (A + A^\top) v, v \rangle
   =
   \frac{1}{2} \lambda \|v\|^2 > 0$.
\end{proof}

\begin{remark}[Necessary condition for local activity]
If equation \eqref{eq:linear-P} with an orthogonal projection $P$ is locally active, then there exists $x \in \R^n$ with $\langle PAx,x \rangle > 0$. This follows directly from Proposition \ref{prop:dissipative} in Appendix \ref{sec:appendix}.
\end{remark}

It is not obvious how the spectrum $\sigma(A)$ of the system matrix $A$ in \eqref{eq:linear-P} is related to local activity. The following two examples indicate a partial answer. A generic statement is formulated in Theorem \ref{thm:generic-local-activity} below.

\begin{example}[In $\R^{n}$ a positive real eigenvalue implies local activity for certain projections]\label{ex:real}
Let $A \in \R^{n \times n}$ and $\lambda > 0$ be an eigenvalue of $A$ with corresponding eigenvector $v \in \R^n$. Let $P \in \R^{n \times n}$ be a projection with $v \in \operatorname{im} P$. Then  \eqref{eq:linear-P} is locally active, in fact, for arbitrary $T > 0$ one can choose
\begin{equation*}\label{eq:locally-active-real}
  u(t)
  \coloneqq
  \rho'\left(\frac{2t}{T}-1\right)e^{\lambda t}v
  \quad \text{for } t\in[0,T],
\end{equation*}
where $\rho\in C_{c}^{\infty}(\mathbb{R})$ denotes the Friedrichs mollifier defined by 
\begin{equation}\label{eq:mollifier}
  \rho(t) \coloneqq
  \begin{cases}
    e^{\frac{1}{t^{2}-1}} & \mbox{ if }-1<t<1,
  \\
    0 & \mbox{ otherwise}.
  \end{cases}
\end{equation}
To see this, we use the fact that $Pu(t) = u(t)$ and compute the solution
\begin{align*}
  x(t) &= \int_{0}^{t} e^{(t-s)A} u(s) \,ds\\
  &=
  \int_{0}^{t} e^{(t-s)\lambda} \rho'\left(\frac{2s}{T}-1\right)e^{\lambda s}v \,ds\\
  &  =
  \frac{T}{2}\rho\left(\frac{2t}{T}-1\right)e^{\lambda t}v
\end{align*}
of \eqref{eq:linear-P}, $x(0)=0$, on $[0,T]$.
We set $q(t)\coloneqq\rho\left(\frac{2t}{T}-1\right)^{2}$ und get
\begin{align*}
  W_{T}(u) 
  &=
  \frac{T}{2}\int_{0}^{T} e^{2\lambda t} 
  \rho\left(\frac{2t}{T}-1\right)\rho'\left(\frac{2t}{T}-1\right) \,dt\, |v|_{\R^n}^{2}
\\
  &=
  \frac{T^2}{8}\int_{0}^{T}e^{2\lambda t}
  q'(t) \,dt\, |v|_{\R^n}^{2}
\\
  &=
  \frac{T^2}{8} \left(e^{2\lambda T}q(T)-q(0)\right) |v|_{\R^n}^{2} -
  \frac{\lambda T^2}{4} \int_{0}^{T} e^{2\lambda t} q(t) \,dt\, |v|_{\R^n}^{2}
\\
  &=
  -\frac{\lambda T^2}{4} \int_{0}^{T}e^{2\lambda t} q(t) \,dt\, |v|_{\R^n}^{2}
  < 0,
\end{align*}
since $q(t)>0$ on $(0,T)$, i.e.\ system \eqref{eq:linear-P} is locally active.
\end{example}

\begin{example}[In $\R^n$ a complex eigenvalue with positive real part implies local activity for certain projections]\label{ex:complex}
Let $A \in \R^{n \times n}$ and $\alpha+\i\beta \in \C$ be an eigenvalue of $A$ with $\alpha > 0$, $\beta \neq 0$, and  corresponding eigenvector $v = v_1 + \i v_2 \in \C^n$ with $v_1, v_2 \in \R^n$. Let $P \in \R^{n \times n}$ be a projection with $v_1, v_2 \in \operatorname{im} P$. Then \eqref{eq:linear-P} is locally active, in fact, there exists $t_0>0$ such that for $T > t_0$ one can choose
\begin{equation}\label{eq:locally-active-complex}
  u(t)
  \coloneqq 
  h'(t)e^{\alpha t}(\sin(\beta t)v_{1}+\cos(\beta t)v_{2})
  \quad \text{for } t\in[0,T],
\end{equation}
where $h(t)\coloneqq\rho\left(\frac{t-t_{0}}{\varepsilon}\right)$ with the Friedrichs mollifier as defined in \eqref{eq:mollifier} and for some suitable $\varepsilon>0$. To see this, we first define
\[
  g(t)
  \coloneqq 
  e^{2\alpha t}|\sin(\beta t)v_{1}+\cos(\beta t)v_{2}|_{\mathbb{R}^{n}}^{2}
  \quad \text{for } t\in\mathbb{R}_{\geq0}
\]
and show that there exists a $t_0>0$ such that $g'(t_0) > 0$. To this end, we distinguish between the following two cases:

Case 1: $\sin(\beta t)v_{1}+\cos(\beta t)v_{2}\ne0$ for each $t\in\mathbb{R}_{\geq0}.$
Then, due to $\beta \neq 0$ and periodicity, we derive $|\sin(\beta t)v_{1}+\cos(\beta t)v_{2}|_{\mathbb{R}^{n}}^{2}\geq c>0$
for some $c>0$ and each $t\in\mathbb{R}_{\geq0}.$ Consequently $g(t)\to\infty$
as $t\to\infty$ and hence, there exists $t_{0}>0$ with $g'(t_{0})>0$.

Case 2: There is $t_1\ge 0$ with $\sin(\beta t_1)v_1 + \cos(\beta t_1)v_2 = 0$. Then $0=g(t_1)=g(t_1+\frac{2\pi}{\beta})$ and $g\ge 0$. Next, $g\ne 0$ on $(t_1,t_1+\frac{2\pi}{\beta})$. Indeed, if $g=0$ on $(t_1,t_1+\frac{2\pi}{\beta})$, then, by peridiodicity, $g=0$ on $\R_{\ge0}$. Thus, \[0=g(0)=|v_2|=g(\pi/(2\beta))=\exp(\alpha\pi/\beta)|v_1|\] contradicting the fact that $v_1+iv_2$ is an eigenvector. So, $g\ge 0$ on $(t_1,t_1+\frac{2\pi}{\beta})$ and $g(t_2)>0$ for some $t_2\in (t_1,t_1+\frac{2\pi}{\beta})$, which eventually implies that there is $t_0\in (t_1,t_1+\frac{2\pi}{\beta})$ with $g'(t_0)>0$.

Let $T > t_0$. Since $g'$ is continuous, there exists $\varepsilon>0$ such that
$g'>0$ on $(t_{0}-\varepsilon,t_{0}+\varepsilon) \subset [0,T].$ Then $h(t) = \rho(\frac{t-t_{0}}{\varepsilon})$ satisfies $h>0$ on $(t_{0}-\varepsilon,t_{0}+\varepsilon)$ and $h=0$ on $[0,T] \setminus (t_{0}-\varepsilon,t_{0}+\varepsilon)$ and hence,
\begin{equation}\label{eq:lemma-gh2}
  \int_{0}^{T} g'(t)h^{2}(t) \,dt 
  =
  \int_{t_{0}-\varepsilon}^{t_{0}+\varepsilon}g'(t)h^{2}(t) \,dt>0.
\end{equation}
Using the facts that $e^{tA}v_{1}=\Re\left(e^{tA}v\right)=\Re\left(e^{t\lambda}v\right)=e^{\alpha t}\left(\cos(\beta t)v_{1}-\sin(\beta t)v_{2}\right)$ and analogously $e^{tA}v_{2}=e^{\alpha t}\left(\cos(\beta t)v_{2}+\sin(\beta t)v_{1}\right)$, and $Pu(t) = u(t)$, we compute the solution $x(t)$ of \eqref{eq:linear-P} which satisfies $x(0)=0$ and get
\begin{align*}
  x(t) 
  &= 
  \int_{0}^{t}e^{(t-s)A}u(s) \,ds 
\\
  &=
  \int_{0}^{t}e^{(t-s)A}h'(s)e^{\alpha s}\sin(\beta s)v_{1} \,ds +
  \int_{0}^{t}e^{(t-s)A}h'(s)e^{\alpha s}\cos(\beta s)v_{2} \,ds
\\
  & =\int_{0}^{t}h'(s)e^{\alpha t}\sin(\beta s)\left(\cos(\beta(t-s))v_{1}-\sin(\beta(t-s))v_{2}\right) \,ds + {} \\
 & \quad+\int_{0}^{t}h'(s)e^{\alpha t}\cos(\beta s)\left(\cos(\beta(t-s))v_{2}+\sin(\beta(t-s))v_{1}\right) \,ds\\
 & =\int_{0}^{t}h'(s) \,ds\; e^{\alpha t}\left(\cos(\beta t)v_{2}+\sin(\beta t)v_{1}\right)\\
 & =h(t)e^{\alpha t}\left(\cos(\beta t)v_{2}+\sin(\beta t)v_{1}\right)
\end{align*}
for each $t\in[0,T]$. Hence, with \eqref{eq:lemma-gh2} we get
\begin{equation*}
W_{T}(u) =\int_{0}^{T}h(t)h'(t)g(t) \,dt =-\frac{1}{2}\int_{0}^{T}g'(t)h^{2}(t) \,dt < 0,
\end{equation*}
i.e.\ system \eqref{eq:linear-P} is locally active.
\end{example}

Examples \ref{ex:real} and \ref{ex:complex} already indicate that we cannot expect the implication
\begin{equation*}
   A \text{ has an eigenvalue with positive real part }
   \Rightarrow
   \text{ system } \eqref{eq:linear-P} \text{ is locally active}
\end{equation*}
to hold without additional assumptions. In Theorem \ref{thm:generic-local-activity} we show that this implication holds generically. As a preparation we need the following lemma which provides an explicit representation of $W_T(u)$ for specific discontinuous two-pulse signals $(u_1,0,\dots,0)$ in the first component.

\begin{lemma}[Scalar two-pulse signals]\label{lem:char_fct}
Consider \eqref{eq:linear-P} and assume that $x_1$ is a port variable, i.e.\ ${\bf e}_1 \in \im P$ and that $A$ is diagonalizable, i.e.\ there exists a non-singular matrix $G = (g_{i \ell})_{i, \ell = 1, \dots, n} \in \C^{n \times n}$ with inverse $H =(h_{\ell j})_{\ell, j = 1, \dots, n} \in \C^{n \times n}$ such that $H A G = \diag(\lambda_1, \dots, \lambda_n)$ where $\lambda_i$ denote the eigenvalues of $A$.
Let $a,b \in \R$, $T > 0$, $k \geq \frac{2}{T}$. Then for $u=(u_1,0,\dots,0) \in L^2([0,T], \R^n)$ with 
\[
  u_1 \coloneqq a \chi_{[0,\frac{1}{k}]}+b\chi_{[T-\frac{1}{k},T]}
\]
we have
\begin{equation}{\label{eq:char}}
  W_T(u)
  = 
  \sum_{\ell=1}^n g_{1\ell} h_{\ell 1}
  \Big(\frac{a^{2}+b^{2}}{\lambda_\ell^{2}}
  \big(e^{\frac{\lambda_\ell}{k}}-1\big)+\frac{ab}{\lambda_\ell^{2}}e^{\lambda_\ell T}
  \big(1-e^{-\frac{\lambda_\ell}{k}}\big)^{2}-\frac{a^{2}+b^{2}}{k \lambda_\ell}\Big).
\end{equation}
\end{lemma}

\begin{proof}
Using the fact that 
\[
  e^{At} = G e^{\diag(\lambda_1, \dots, \lambda_n) t} H
  = \Big(\sum_{\ell=1}^n g_{i\ell} h_{\ell j}  e^{\lambda_\ell t} \Big)_{i,j=1,\dots,n}
\]
we have for $x = (x_1,0,\dots,0) \in \R^n$ 
\[
  e^{A(t-\tau)} x
  =
  \sum_{i=1}^n \Big(\sum_{\ell=1}^n g_{i\ell} h_{\ell 1} e^{\lambda_\ell (t-\tau)} x_1\Big) \e_i.
\]
Let $u=(u_1,0,\dots,0) = u_1 \e_1 \in C([0,T],\R^n)$ be arbitrary. By assumption $x_1$ is a port variable and therefore $P u = u$. The solution $x$ of \eqref{eq:linear-P}, $x(0)=0$, is $x(t) = \int_0^t e^{A(t-\tau)} u_1(\tau) \e_1 \,d\tau$ and hence
\begin{align}\label{corerepresentationformula}
  W_T(u_1 \e_1) 
  &= 
  \int_0^T \langle x(t),P u_1(t) \e_1 \rangle \,dt
  =
  \int_0^T \Big\langle \int_0^t e^{A(t-\tau)} u_1(\tau) \e_1 \,d\tau,
  u_1(t) \e_1
  \Big\rangle \,dt 
  \nonumber
  \\ 
   &=
  \int_0^T \Big\langle \int_0^t \sum_{i=1}^n 
  \Big(\sum_{\ell=1}^n g_{i\ell} h_{\ell 1} e^{\lambda_\ell (t-\tau)} u_1(\tau)\Big) \e_i \, d\tau, 
  u_1(t) \e_1\Big\rangle \, dt 
  \nonumber
  \\ 
   &=
   \int_0^T \int_0^t \sum_{\ell=1}^n 
   g_{1\ell} h_{\ell 1} e^{\lambda_\ell (t-\tau)} u_1(\tau) u_1(t) \,d\tau \,dt.
\end{align} 
Chosing $u_1 \coloneqq a \chi_{[0,\frac{1}{k}]}+b\chi_{[T-\frac{1}{k},T]}$, a direct computation shows 
for arbitrary $\lambda \in \C$ and $t \in [0,T]$ that
$\int_0^T \int_0^t e^{\lambda(t-\tau)} u_1(\tau)u_1(t) \,d\tau \,dt
=
\frac{a^{2}+b^{2}}{\lambda^{2}}(e^{\frac{\lambda}{k}}-1)+\frac{ab}{\lambda^{2}}e^{\lambda T}
(1-e^{-\frac{\lambda}{k}})^{2}-\frac{a^{2}+b^{2}}{k \lambda}$ and \eqref{eq:char} follows from \eqref{corerepresentationformula}.
\end{proof}


The main result of this section is that, in the generic case, instability of the linear system ${\dot x} = Ax$ implies local activity.

\begin{theorem}[Instability generically implies local activity]\label{thm:generic-local-activity}
Generically, if system \eqref{eq:linear-P} in $\R^n$ has an eigenvalue with positive real part, then it is locally active. More precisely, for every projection matrix $P \in \R^{n \times n}$, $P \neq 0$, there exists an open and dense set $\mathcal{M} \subseteq \R^{n \times n}$ such that the following implication holds
\begin{equation*}
   A \in \mathcal{M}
   \;\wedge\;
   \Re(\sigma(A)) > 0
   \;\;\Rightarrow\;\;
   \exists T \geq 2 \,\exists u \in C([0,T], \R^n)  \colon W_T(u) < 0.
\end{equation*}
\end{theorem}

\begin{proof}We divide the proof into four steps. In the first step we transform system \eqref{eq:linear-P} by an orthogonal coordinate transformation such that $\e_1$ is contained in the image of $P$. Then we compute $W_T(u)$ for two possible choices of  $u=(u_1,0,\dots,0) \in L^2([0,T], \R^n)$ according to Lemma \ref{lem:char_fct} in Step 2. Step 3 yields the estimate $W_T(u) < 0$ for the appropriate choice of $u$ and $T$ large enough. In Step 4 we approximate $u \in L^2([0,T], \R^n)$ by a continuous function in $C([0,T], \R^n)$.

{\bf Step 1:} W.l.o.g.\ $x_1$ is a port variable. More precisely, since $P \neq 0$ there exists a $v \in \im P \setminus \{0\}$ and an orthogonal matrix $Q \in \R^{n \times n}$ with $Qv = \e_1$. Using the nomenclature of Appendix \ref{subset:genericity}, Lemma \ref{cor:product_generic_nonzero} implies that the set
\begin{multline*}
   {\cal M}
   \coloneqq
   \{ Q^{-1}MQ \in \R^{n \times n} \,:\, M \in \R^{n \times n}, 0 \not\in \sigma(M),\\ |\sigma(M)|=n, \exists \lambda\in \sigma(M): \Re \lambda>\max \big(\Re (\sigma(A)\setminus \{\lambda,\overline{\lambda}\})\big),\\ 
   \text{ and } 
   g_{11}^M h_{11}^M \neq 0 \}
\end{multline*}
is open and dense in $\R^{n \times n}$. Let $A \in {\cal M}$. Then $x \mapsto \widetilde{x} = Q x$ transforms system \eqref{eq:linear-P} into
\begin{equation*}
  \dot{\widetilde{x}}(t) = \widetilde{A} \widetilde{x}(t) + \widetilde{P} \widetilde{u}(t)
\end{equation*}
with $\widetilde{A} = Q A Q^{-1}$, $\widetilde{P} = Q P Q^{-1}$, $\widetilde{u} = Q u$ and $\widetilde{x}_1$ is a port variable, since $\widetilde{P} \e_1 = Q P Q^{-1} \e_1 = Q P v = Q v = \e_1$. Moreover, $g_{11}^{\widetilde{A}} h_{11}^{\widetilde{A}} \neq 0$. For notational convenience, we omit the tilde and write again $x, u, A, P$ instead of $\widetilde{x}, \widetilde{u}, \widetilde{A}, \widetilde{P}$.

{\bf Step 2:} For $a \in \{-1,1\}$ define $u=(a \chi_{[0,1]} + \chi_{[T-1,T]},0,\dots,0) \in L^2([0,T], \R^n)$ for $T \geq 2$. We will fix $a$ in Step 3 such that $W_T(u) < 0$. Let $\lambda_1,\ldots,\lambda_n$ be an enumeration of $\sigma(A)$ such that $\Re \lambda_1=\max_{\ell\in \{1,\ldots,n\}}\Re\lambda_\ell$. For $\lambda \in \sigma(A)$, define $c_{\lambda} \coloneqq g_{1\ell}^A h_{\ell 1}^A \frac{1}{\lambda_\ell^2} (1 - e^{-\lambda_\ell})^2$, with $\ell \in \{1,\dots,n\}$ such that $\lambda = \lambda_\ell$. We note that $(g_{1\ell}^{A}, \dots, g_{n \ell}^{A})^T$ is the $\ell$-th column of $G(A)$ (see Appendix \ref{subset:genericity} for the corresponding notation) being the eigenvector corresponding to $\lambda_\ell$. By our choice of $A$, $c_{\lambda_1} \neq 0$. Lemma \ref{lem:char_fct} for $k=b=1$ yields
\begin{equation}\label{eq:W_T:generic}
  W_T(u)
  =
  \eta + a \sum_{\lambda \in \sigma(A)} c_\lambda e^{\lambda T},
\end{equation}
where we used the abbreviation $\eta \coloneqq \sum_{\ell=1}^n g_{1\ell}^A h_{\ell 1}^A (\frac{2}{\lambda_\ell^{2}}(e^{\lambda_\ell}-1)-\frac{2}{\lambda_\ell})$.

{\bf Step 3:} We rewrite \eqref{eq:W_T:generic}
\begin{align*}
  W_T(u) &= e^{(\Re \lambda_1) T}
  \Big( e^{-(\Re \lambda_1)T} \eta + a \sum_{\lambda \in \sigma(A)} c_\lambda e^{(\lambda - \Re \lambda_1) T} \Big)
\\
  &= e^{(\Re \lambda_1) T}
  \Big( e^{-(\Re \lambda_1)T} \eta + a \sum_{\lambda \in \sigma(A) \setminus\{\lambda_1, \overline{\lambda}_1\}} 
  c_\lambda e^{(\lambda - \Re \lambda_1) T} +
\\ & \quad +
  |\{\lambda_1, \overline{\lambda}_1\}| a \Re(c_{\lambda_1} e^{i (\Im \lambda_1) T})\Big),
\end{align*}
where we used the fact that $\sum_{\lambda \in \{\lambda_1, \overline{\lambda}_1\}} c_\lambda e^{(\lambda - \Re \lambda_1)T}$ equals $2 \Re(c_{\lambda_1}a e^{i (\Im \lambda_1) T})$ if $\lambda_1 \not\in \R$, and it
equals $c_{\lambda_1} \big(= \Re(c_{\lambda_1} e^{i (\Im \lambda_1) T})\big)$ if $\lambda_1 \in \R$. Note that for the first two terms in brackets we have
\begin{equation*}
  \lim_{T \to \infty} e^{-(\Re \lambda_1)T} \eta
  =
  \lim_{T \to \infty} a \sum_{\lambda \in \sigma(A) \setminus\{\lambda_1, \overline{\lambda}_1\}} 
  c_\lambda e^{(\lambda - \Re \lambda_1) T}
  = 
  0,
\end{equation*}
by the choice of $\lambda_1$, and for the third term, since $c_{\lambda_1} \neq 0$,
\begin{equation*}
  \limsup_{T \to \infty} \big| |\{\lambda_1, \overline{\lambda}_1\}| a \Re(c_{\lambda_1} e^{i (\Im \lambda_1) T}) \big|
  \eqqcolon m > 0.
\end{equation*}
We choose and fix $a \in \{-1,1\}$ such that for every $T_0 \geq 2$ there exists $T \geq T_0$ with $|\{\lambda_1, \overline{\lambda}_1\}| a \Re(c_{\lambda_1} e^{i (\Im \lambda_1) T}) < -\frac{m}{2}$. As a consequence there exists $T \geq 2$ with $W_T(u) < 0$. 

{\bf Step 4:} Using the fact that $C([0,T], \R)$ is dense in $L^2([0,T], \R)$ and $W_T : L^2([0,T], \R^n) \rightarrow \R$ is continuous, we can also find a $u=(u_1,0,\dots,0) \in C([0,T], \R^n)$ with $W_T(u) < 0$. Transforming back to the original coordinate system with $x \mapsto Q^{-1} x$ yields the statement of the theorem for system \eqref{eq:linear-P}.
\end{proof}

\begin{remark}
We conjecture that the statement of Theorem 7 is also true for genericity in the measure-theoretic sense. For results in this direction, see Arnold \cite[\S 30.H]{Arnold1988}.
\end{remark}

\section{Nonlinear systems and local activity}\label{sec:nonlinear}

Let $n \in \N$, $P \in \R^{n \times n}$ be a projection and let $f, D : \R^n \rightarrow \R^n$ be $C^1$ functions with $f(x_0) - P D(x_0) = 0$ for some $x_0 \in \R^n$. Consider the differential equation 
\begin{equation}\label{eq:nonlinear1}
  \dot x(t) = f(x(t)) - P D(x(t))
\end{equation}
with equilibrium $x_0$.
In this section we illustrate how to associate with \eqref{eq:nonlinear1} the linear system
\begin{equation}\label{eq:linear1}
  \dot x(t) = \frac{d f}{d x}(x_0) x(t) + P u(t)
\end{equation}
for $u  \in C(\R,\R^n)$. It is not yet fully understood how complexity of \eqref{eq:nonlinear1} might be induced from local activity of \eqref{eq:linear1}. To illustrate this, consider the simplest situation for $n=1$ and $P = 1 \in \R^{1 \times 1}$. Note that with the abbreviation $\lambda \coloneqq \tfrac{d f}{d x}(x_0)$ we get 
$x(t) = \int_0^t e^{\lambda(t-s)} u(s) \,ds$ as the solution of \eqref{eq:linear1}, $x(0)=0$, and hence \eqref{eq:linear1} is locally active if for some $u : [0,T] \rightarrow \R$ with $T > 0$ the inequality
\begin{equation}\label{eq:local-activity-simple}
  \int_0^T  \int_0^t e^{\lambda(t-s)} u(s) \,ds u(t) \,dt < 0
\end{equation}
holds. In Theorem \ref{thm:local-activity-allport} we showed that this is equivalent to the condition $\lambda > 0$. In Example \ref{ex:real}, we have seen that if $\lambda > 0$ then \eqref{eq:local-activity-simple} is satisfied for any $T > 0$ with $u(t) \coloneqq \rho'(2t/T - 1) e^{\lambda t}$ for $t \in [0,T]$, where $\rho\in C_{c}^{\infty}(\mathbb{R})$ denotes the Friedrichs mollifier \eqref{eq:mollifier}. The differential equation \eqref{eq:nonlinear1} for $n=1$ and $P=1$ takes the form
\begin{equation}\label{eq:scalar-nonlinear}
   \dot x(t) = f(x(t)) - D(x(t))
\end{equation}
and by the theorem of linearized asymptotic stability \cite[Theorem 2.77, p.\ 183]{Chicone2006}, the equilibrium $x_0$ of \eqref{eq:scalar-nonlinear} is asymptotically stable if its linearization
\begin{equation}\label{eq:scalar-linear}
   \dot x(t) = \Big[\frac{df}{dx}(x_0) - \frac{d D}{dx} (x_0)\Big]x(t)
\end{equation}
is exponentially stable, i.e.\ if $\gamma \coloneqq \frac{df}{dx}(x_0) - \frac{d D}{dx}(x_0) < 0$ and unstable if $\gamma > 0$. It might therefore happen that the linear differential equation $\dot x(t) = \frac{df}{dx}(x_0) x(t)$ is stable, i.e.\ $\lambda < 0$, whereas the full linearization \eqref{eq:scalar-linear} which also involves the $D$-term is unstable, i.e.\ $\gamma > 0$. In summary, in the scalar case with $P=1$, \eqref{eq:linear1} is locally passive if and only if $\lambda \leq 0$ and in this case it might still happen that \eqref{eq:scalar-linear} is unstable, namely if $\frac{d D}{dx} (x_0) < \lambda \leq 0$.

An interesting and counter-intuitive case occurs for $n \geq 2$, if $x_0$ is an asymptotically stable equilibrium of the ``kinetic part'' $\dot x(t) = f(x(t))$ of \eqref{eq:scalar-nonlinear}, $-D$ is a ``diffusion'' or ``dissipation term'' with $\langle x, D(x) \rangle \geq 0$ for all $x \in \R^n$ and the equilibrium $x_0$ of \eqref{eq:scalar-linear} is unstable, i.e.\ a ``dissipation-induced destabilization'' occurs (for a discussion of this effect for various classes of differential equations see e.g.\ \cite{KirillovVerhulst10} and the references therein). Consider e.g.\ \eqref{eq:nonlinear1} for $n=2$, $P = I$, $f(x) = Ax$,
\begin{equation*}
   A = \begin{pmatrix}
      -1 & 10 \\ 0 & -2 
   \end{pmatrix}
   \qquad \text{and} \qquad
   D = \begin{pmatrix}
      1 & -1 \\ -1 & 1 
   \end{pmatrix}.
\end{equation*}
Then the symmetric matrix $-D$ is dissipative, since the eigenvalues of $D$ are $0$ and $2$. Moreover, $\dot x(t) = A x(t)$ is asymptotically stable, since the eigenvalues of $A$ are $-1$ and $-2$. However, \eqref{eq:scalar-linear} (or equivalently \eqref{eq:nonlinear1} due to linearity) is of the form $\dot x(t) = [A - D] x(t)$ which is unstable, since $A - D$ has a positive eigenvalue.

We briefly recall two examples from \cite{Chua05, ItohChua07}, a FitzHugh--Nagumo equation with dissipation and a discrete reaction-diffusion equation, before we propose a ``linearization'' scheme and related open questions on local activity and complexity.

\begin{example} [FitzHugh--Nagumo equation with dissipation {\cite{ItohChua07}}]\label{ex:FitzNugh-Nagumo}
Consider the FitzHugh--Nagumo equation with a dissipation term 
\begin{equation}\label{eq:FitzNugh-Nagumo}
  \begin{array}{rcl}\displaystyle
  \frac{dx}{dt} & = & -y - f(x) - \mu x
\\[2ex] \displaystyle
  \frac{dy}{dt} & = & \xi(x - \beta y + \gamma)
  \end{array}
\end{equation}
with $f(x) = (x^3/3) - x$, $\beta = 1.28$, $\gamma = 0.12$, $\xi = 0.1$ and a dissipation coefficient $\mu > 0$. For small $\mu$ equation \eqref{eq:FitzNugh-Nagumo} has an equilibrium $(x_d, y_d) = (x_d(\mu), y_d(\mu))$ which undergoes a Hopf bifurcation
at $\mu \approx 0.05$ with $(x_d, y_d) \approx (-0.9083, -0.6159)$ \cite[Section 4.1]{ItohChua07}.

In \cite[Section 4.1]{ItohChua07} for an arbitrary solution $(x(t),y(t))$ of \eqref{eq:FitzNugh-Nagumo} the dissipation term $-\mu x(t)$ is interpreted as and replaced by an \emph{input term} $-\mu x_d(\mu) + \delta i(t)$ for a general \emph{input function} $t \mapsto \delta i(t)$ and we arrive at an associated family of \emph{forced FitzHugh--Nagumo equations}
\begin{equation}\label{eq:FitzNugh-Nagumo2}
  \begin{array}{rcl}\displaystyle
  \frac{dx}{dt} & = & -y - f(x) - \mu x_d(\mu) + \delta i(t)
\\[2ex] \displaystyle
  \frac{dy}{dt} & = & \xi(x - \beta y + \gamma)
  \end{array}
\end{equation}
with input functions $\delta i \in C(\R_{\geq 0},\R)$.

Approximating $f$ by its Taylor expansion of order 1 in the equilibrium $(x_d, y_d)$ of \eqref{eq:FitzNugh-Nagumo} yields an associated class of linear differential equations in $(\delta x, \delta y)$-variables
\begin{equation}\label{eq:FitzNugh-Nagumo3}
  \begin{array}{rcl}\displaystyle
  \frac{d(\delta x)}{dt} & = & \displaystyle - \delta y - \frac{df\big( x_d(\mu) \big)}{dx} \delta x + \delta i(t)
\\[2ex] \displaystyle
  \frac{d(\delta y)}{dt} & = & \xi(\delta x - \beta \delta y)
  \end{array}
\end{equation}
with $\frac{df( x_d(\mu) )}{dx} = x_d(\mu)^2 - 1$ and $\delta i \in C(\R_{\geq 0},\R)$.
\end{example}

\begin{example}[Discrete reaction-diffusion equation {\cite[Equation (1)]{Chua05}}]\label{ex:reaction-diffusion}
Consider the discrete reaction-diffusion equation given on an integer grid $Z \subseteq \Z^d$ for $d \in \{1,2,3\}$, i.e.\ for every ${\mathbf r} \in Z$ we have
\begin{equation}\label{eq:discr-reaction-diffusion}
  \begin{array}{rcl}\displaystyle
    \frac{d V_1({\mathbf r})}{dt} & = & f_1(V_1({\mathbf r}), \dots, V_n({\mathbf r})) + D_1 \nabla^2 V_1({\mathbf r})
  \\
    & \vdots
  \\ \displaystyle
    \frac{d V_m({\mathbf r})}{dt} & = & f_m(V_1({\mathbf r}), \dots, V_n({\mathbf r})) + D_m \nabla^2 V_m({\mathbf r})
  \\[1ex] \displaystyle
    \frac{d V_{m+1}({\mathbf r})}{dt} & = & f_{m+1}(V_1({\mathbf r}), \dots, V_n({\mathbf r}))
  \\
    & \vdots
  \\ \displaystyle
    \frac{d V_n({\mathbf r})}{dt} & = & f_n(V_1({\mathbf r}), \dots, V_n({\mathbf r}))
  \end{array}
\end{equation}
where $V_1({\mathbf r}), \dots, V_n({\mathbf r})$ denote the state variables of a ``reaction cell'' located at the grid point ${\mathbf r} \in Z$. $D_1, \dots, D_m > 0$ denote the diffusion coefficients associated with the first $m$ state variables and $\nabla^2 V_i({\mathbf r})$ for $i = 1, \dots, m$, denotes the discretized Laplace operator on $Z$. E.g.\ for $d=2$ write ${\mathbf r} = (j,k) \in Z \subset \Z^2$, then $\nabla^2 V_i(j,k) \coloneqq V_i(j+1,k) + V_i(j-1,k) + V_i(j,k+1) + V_i(j,k-1) - 4V_i(j,k)$ for $i = 1, \dots, m$. For an $N \times N$ array $Z \coloneqq \{1,2,\dots,N\}^2$ with $N \in \N$ one could impose Dirichlet, Neumann or toroidal boundary conditions on \eqref{eq:discr-reaction-diffusion} as described in \cite[Section 2]{ItohChua07}.

Note that any interaction between two cells in $Z$ can come only from the diffusion terms $D_i \nabla^2 V_i({\mathbf r})$ for $i=1,\dots,m$. With the abbreviations ${\mathbf V_a} \coloneqq (V_1,\dots,V_m)^\top$, ${\mathbf V_b} \coloneqq (V_{m+1},\dots,V_n)^\top$, ${\mathbf f_a} \coloneqq (f_1,\dots,f_m)^\top$, ${\mathbf f_b} \coloneqq (f_{m+1},\dots,f_n)^\top$, ${\mathbf D} \coloneqq \operatorname{diag} (D_1,\dots,D_m) \in \R^{m \times m}$, $\nabla^2 {\mathbf V_a} \coloneqq (\nabla^2 V_1,\dots,\nabla^2 V_m)^\top \in \R^m$, system \eqref{eq:discr-reaction-diffusion} can be rewritten as
\begin{equation}\label{eq:discr-reaction-diffusion2}
  \begin{array}{rcl}\displaystyle
    \dot {\mathbf V}_a({\mathbf r}) & = & {\mathbf f_a}({\mathbf V}_a({\mathbf r}), {\mathbf V}_b({\mathbf r}))
      + {\mathbf D} \nabla^2 {\mathbf V_a}({\mathbf r})
  \\[1ex] \displaystyle
    \dot {\mathbf V}_b({\mathbf r}) & = & {\mathbf f_b}({\mathbf V}_a({\mathbf r}), {\mathbf V}_b({\mathbf r}))
  \end{array}
\end{equation}
As in \cite[Section 5]{ItohChua07} consider one single cell in the plane, i.e.\ $Z = \{(1,1)\}$, under fixed (Dirichlet) boundary conditions $V_i(0,1)=V_i(1,0)=V_i(2,1)=V_i(1,2)=0$. Then the discrete Laplacian becomes $D_i \nabla^2 V_i(1,1) = -4 D_i V_i(1,1)$ and \eqref{eq:discr-reaction-diffusion2} on the single cell ${\mathbf r} = (1,1)$ simplifies to
\begin{equation}\label{eq:discr-reaction-diffusion2b}
  \begin{array}{rcl}\displaystyle
    \dot {\mathbf V}_a & = & {\mathbf f_a}({\mathbf V}_a, {\mathbf V}_b)
      + {\mathbf D} \nabla^2 {\mathbf V_a}
  \\[1ex] \displaystyle
    \dot {\mathbf V}_b & = & {\mathbf f_b}({\mathbf V}_a, {\mathbf V}_b)
  \end{array}
\end{equation}
(compare with \cite[Formula (109)]{ItohChua07} in case $D_i \coloneqq \mu/4$, $i = 1, \dots,m$, for some $\mu > 0$).

Assume that \eqref{eq:discr-reaction-diffusion2b} has an equilibrium $(\overline{{\mathbf V}}_a, \overline{{\mathbf V}}_b)$ and define $\overline{{\mathbf I}}_a \coloneqq {\mathbf D} \nabla^2 \overline{{\mathbf V}}_a$. Similar as in \cite{Chua05, ItohChua07} for an arbitrary solution $({\mathbf V}_a(t),{\mathbf V}_b(t))$ of \eqref{eq:discr-reaction-diffusion2b} the diffusion term ${\mathbf D} \nabla^2 {\mathbf V_a}(t)$ is interpreted as and replaced by an \emph{interaction term} $\overline{{\mathbf I}}_a + {\mathbf i}_a(t)$ for a general \emph{input function} $t \mapsto {\mathbf i}_a(t) \in \R^m$ and we arrive at an associated family of \emph{forced cell kinetic equations}
\begin{equation}\label{eq:discr-reaction-diffusion3}
  \begin{array}{rcl}\displaystyle
    \dot {\mathbf V}_a & = & {\mathbf f_a}({\mathbf V}_a, {\mathbf V}_b)
      + \overline{{\mathbf I}}_a + {\mathbf i}_a(t)
  \\[1ex] \displaystyle
    \dot {\mathbf V}_b & = & {\mathbf f_b}({\mathbf V}_a, {\mathbf V}_b)
  \end{array}
\end{equation}
with input functions ${\mathbf i}_a \in C(\R_{\geq 0},\R^m)$.

Approximating $({\mathbf f_a},{\mathbf f_b})$ by its Taylor expansion of order 1 in the equilibrium $(\overline{{\mathbf V}}_a, \overline{{\mathbf V}}_b)$ of \eqref{eq:discr-reaction-diffusion2b} yields an associated class of linear differential equations
\begin{equation}\label{eq:discr-reaction-diffusion4}
  \begin{array}{rcl}\displaystyle
    \dot {\mathbf v}_a & = & A_{11} {\mathbf v}_a +A_{12} {\mathbf v}_b + {\mathbf i}_a(t)
  \\[1ex] \displaystyle
    \dot {\mathbf v}_b & = & A_{21} {\mathbf v}_a +A_{22} {\mathbf v}_b
  \end{array}
\end{equation}
with ${A_{11} \;\; A_{12} \choose A_{21} \;\; A_{22}} = \frac{\partial ({\mathbf f_a},{\mathbf f_b})}{\partial ({\mathbf V}_a,{\mathbf V}_b)}(\overline{{\mathbf V}}_a, \overline{{\mathbf V}}_b)$ and ${\mathbf i}_a \in C(\R_{\geq 0},\R^m)$.
\end{example}

Examples \ref{ex:FitzNugh-Nagumo} and \ref{ex:reaction-diffusion} follow the same scheme of starting with a nonlinear differential equation and then associating a class of linear systems \eqref{eq:FitzNugh-Nagumo3} and \eqref{eq:discr-reaction-diffusion4}, respectively, for which local activity can be checked depending on the parameters. To formalize this scheme let $n \in \N$, $P \in \R^{n \times n}$ be a projection and let $f, D : \R^n \rightarrow \R^n$ be $C^1$ functions with $f(x_0) - P D(x_0) = 0$ for some $x_0 \in \R^n$ and $\langle x, D(x) \rangle \geq 0$ for all $x \in \R^n$. Consider the following 

\paragraph{(A) Differential equation with dissipation or diffusivity term}
\begin{equation}\label{eq:nonlinear}
  \dot x(t) = f(x(t)) - P D(x(t)).
\end{equation}
In case $P$ projects to some of the coordinate components of $x = (x_1, \dots, x_n)$ then those variables are called \emph{port variables} \cite[Section 1]{Chua05}. 

Example \ref{ex:FitzNugh-Nagumo}, equation \eqref{eq:FitzNugh-Nagumo}, is of the form \eqref{eq:nonlinear} with $n=2$, $P = {1 \; 0 \choose 0 \; 0}$, $f(x) = (x_1 -x_2 - x_1^3/3,\xi(x_1 - \beta x_2 + \gamma))^\top$, $D(x) = \mu x$, and the first component $x_1$ is a port variable.

Example \ref{ex:reaction-diffusion}, equation \eqref{eq:discr-reaction-diffusion2b}, is of the form \eqref{eq:nonlinear} with the projection matrix $P = \operatorname{diag}(1, \dots, 1, 0, \dots, 0) \in \R^{n \times n}$ which projects on the first $m$ components $(x_1,\dots,x_m)^\top = {\mathbf V}_a$ of $x = (x_1,\dots,x_n)^\top = ({\mathbf V}_a,{\mathbf V}_b)^\top \in \R^n$, $f(x) = ({\mathbf f_a}({\mathbf V}_a, {\mathbf V}_b),{\mathbf f_b}({\mathbf V}_a, {\mathbf V}_b))^\top$, $D(x) = (-{\mathbf D} \nabla^2 {\mathbf V_a}, 0)^\top$ and the first $m$ components ${\mathbf V}_a$ of $x$ are port variables.

In a next step the term $- P D(x)$ is replaced by a general perturbation of $- P D(x_0)$, i.e.\ \eqref{eq:nonlinear} is replaced by an associated class of differential equations.

\paragraph{(B) Associated class of perturbed differential equations}

\begin{equation}\label{eq:nonlinear2}
  \dot x(t) = f(x(t)) - P D(x_0) + Pu(t)
\end{equation}
for arbitrary \emph{perturbations} $u$ in a given subset $\mathcal I$ of the space of locally integrable functions $u : \R \rightarrow \R^n$.

Example \ref{ex:FitzNugh-Nagumo}, equation \eqref{eq:FitzNugh-Nagumo2}, is of the form \eqref{eq:nonlinear2} with $u = (\delta i,0)^\top$.

Example \ref{ex:reaction-diffusion}, equation \eqref{eq:discr-reaction-diffusion3}, is of the form \eqref{eq:nonlinear2} with $u = ({\mathbf i}_a,0)^\top$.

Next we use the fact that $f(x_0) - P D(x_0) = 0$ and ``linearize'' \eqref{eq:nonlinear2} at the equilibrium $x_0$ of \eqref{eq:nonlinear} in the sense that $f(x(t))$ is replaced by its Taylor expansion of order 1 in $x_0$. 

\paragraph{(C) Associated class of linear differential equations}

\begin{equation}\label{eq:linear}
  \dot x(t) = \tfrac{d f}{d x}(x_0) x(t) + P u(t)
\end{equation}
for $u  \in \mathcal I$. 
Example \ref{ex:FitzNugh-Nagumo}, equation \eqref{eq:FitzNugh-Nagumo3}, and Example \ref{ex:reaction-diffusion}, equation \eqref{eq:discr-reaction-diffusion4}, are of the form \eqref{eq:linear}.

In step (B) a whole class of differential equations \eqref{eq:nonlinear2} is associated to a single differential equation \eqref{eq:nonlinear}. It would be interesting to answer the following question.
\begin{enumerate}
   \item[Q1:] How are the solutions of \eqref{eq:nonlinear} and \eqref{eq:nonlinear2} related? More precisely, characterize the set of perturbations $u$ for which those two systems are topologically conjugate (see e.g.\ \cite{Popescu2004} and the references therein).
\end{enumerate}

In step (C) a linearization of the ``kinetic'' part $\dot x(t) = f(x(t))$ at $x_0$ is applied to the class of nonautonomous equations \eqref{eq:nonlinear2} although $x_0$ is not necessarily an equilibrium of $f$ nor of the nonautonomous equation \eqref{eq:nonlinear2} (i.e.\ $f(x_0) - PD(x_0) + P u(t) = 0$ is not satisfied for all $t \in \R$). One could ask the following question.
\begin{enumerate}
   \item[Q2:] How are the solutions of \eqref{eq:nonlinear2} and \eqref{eq:linear} related? Are those equations topologically conjugate close to $x_0$ resp.\ $0 \in \R^n$?
\end{enumerate}

In Theorem \ref{thm:generic-local-activity} we have shown that if $P \neq 0$ then generically instability implies local activity. The following question arises.
\begin{enumerate}
   \item[Q3:] Characterize those locally active systems \eqref{eq:linear-P} which are asymptotically stable.
\end{enumerate}

Complexity encompasses definitely more phenomena than merely instability of equilibria which are the focus of this paper. Chua \cite[Theorem 3]{Chua05} proves a characterization of local activity for projections of the form $P = \diag(1,0,\dots,0)$ in terms of four properties of an appropriate Laplace transform, the so-called complexity function with respect to the ``\emph{input-signal-port}'' \cite[Formula (18)]{Chua05}. For the example of the FitzHugh-Nagumo equation \cite[Section 4.1]{ItohChua07}, which is described by \eqref{eq:FitzNugh-Nagumo} and its associated class of linear differential equations in $(\delta x, \delta y)$-variables \eqref{eq:FitzNugh-Nagumo3}, the complexity function is of the form \cite[Formulas (62) and (63)]{ItohChua07}
\[
   Y(s) = \frac{s^2 + (\xi \beta + x_d(\mu)^2 - 1)s + \xi \beta (x_d(\mu)^2 -1) + \xi}{s + \xi \beta}
\]
for those $s \in \C$ for which it is defined. According to Theorem 3 in \cite{Chua05}, system \eqref{eq:FitzNugh-Nagumo3} is locally active if and only if at least one of the following conditions holds:
\begin{itemize}
   \item[(i)] $Y(s)$ has a pole in $\Re[s] > 0$.
   
   \item[(ii)] $Y(s)$ has a multiple pole on the imaginary axis.
   
   \item[(iii)] $Y(s)$ has a simple pole $s = i \omega_p$ on the imaginary axis and 
      $\lim_{s \to i \omega_p} (s - i\omega_p) Y(s)$ is either a negative real number, or a complex number. 
   
   \item[(iv)] $\Re[Y(i\omega)] < 0$ for some $\omega \in (-\infty, \infty)$.
\end{itemize}
Chua calls a differential equation with corresponding complexity function $Y(s)$ at the \emph{edge of chaos} if condition (iv) is satisfied and conditions (i), (ii) and (iii) are not satisfied.  A system at the edge of chaos does not necessarily show complexity, but a lack of the edge of chaos property is an obstruction to the emergence of complexity (e.g., the emergence of non-homogeneous static or dynamic patterns) \cite{Chua98, Chua05}. Hence the following informal questions need to be formulated more precisely and related to the concept of local activity.
\begin{enumerate}
   \item[Q4:] What is the generalization of Theorem 3 in \cite{Chua05}, i.e.,
      how can the concept of local activity be characterized by the complexity function for 
      arbitrary projections $P$ on a Hilbert space?
   \item[Q5:] How can one characterize edge of chaos for other classes of systems like 
      lattice dynamical systems or 1-dimensional cellular automata, 
      e.g.\ defined by a ring of binary cells, and what does it mean for the emergence 
      of non-homogeneous static or dynamic patterns?
\end{enumerate}

Of course this list of questions is by no means complete and only intended to fuel a fruitful discussion on the relation between local activity and the emergence of complexity.

\section{Appendix}\label{sec:appendix}

\subsection{Genericity}\label{subset:genericity}

The matrices in ${\R}^{n \times n}$ for which several eigenvalues attain the maximum real part form a closed codimension-one manifold-like object, in mathematical terms, a closed semi-algebraic subvariety of codimension one  \cite[\S30.H]{Arnold1988}. As a consequence, generically, only one real eigenvalue or only one pair of complex conjugate eigenvalues attains the maximum or dominating real part. In the next lemma we show in addition that it is also generic that a matrix is invertible and has distinct eigenvalues.

\begin{lemma}[Genericity of non-singular matrices with separated spectrum]\label{thm:matrix_generic}
The set ${\cal N}$ of invertible matrices with distinct eigenvalues and an eigenvalue with dominating real part, i.e.
\begin{multline*}
 {\cal N}\coloneqq \{A \in \R^{n \times n} \,:\, 0 \not\in \sigma(A),\\ \exists \lambda \in \sigma(A): \Re \lambda>\max\big(\Re(\sigma(A)\setminus \{\lambda,\overline{\lambda}\})\big) \text{ and } |\sigma(A)|=n \}
\end{multline*}
is open and dense in $\R^{n \times n}$.
\end{lemma}

\begin{proof}
$(i)$ ${\cal N}$ is open: The statement rests on the fact that $\sigma\colon\R^{n\times n}\ni A\mapsto \sigma(A)$ is a continuous function with respect to the Hausdorff metric, see \cite[Theorem 3.1.2, p.\ 45]{Ortega1987}. Let $A \in {\cal N}$ with $\lambda_1\in \sigma(A)$ be such that $\Re \lambda_1=\max(\Re\sigma(A))$. Denote $\delta_1 \coloneqq \min \{ |\lambda - \mu| \,:\, \lambda, \mu \in \sigma(A)\cup\{0\},\lambda\ne\mu \} > 0$ and $\delta_2\coloneqq \dist\Big(\Re\lambda_1,\big(\Re(\sigma(A)\setminus \{\lambda_1,\overline{\lambda_1}\})\big)\Big)>0$ as well as $\delta\coloneqq \min\{\delta_1,\delta_2\}/2$.  By the continuity of $\sigma$, the set 
\[
   \sigma^{-1}\Big[ B_\textnormal{H}(\sigma(A),\delta)\Big]
\]
is open, contains $A$, and, by the choice of $\delta$, the set is a subset of $\mathcal N$, where $B_\textnormal{H}(\sigma(A),\delta)$ denotes the ball containing all compact sets $K\subset \mathbb C$ such that $d_\textnormal{H}(\sigma(A),K)<\delta$ ($d_\textnormal{H}$ the Hausdorff distance). For the latter we observe the following. Let $B\in \R^{n\times n}$ such that $d_\textnormal{H}(\sigma(A),\sigma(B))<\delta$. Then for each $\lambda\in \sigma(A)$ there is a unique $\mu\in\sigma(B)$ with $|\lambda-\mu|<\delta$. The existence follows from the definition of the Hausdorff distance. For the uniqueness note that  $B(\lambda,\delta)\cap B(\kappa,\delta)=\emptyset$ for all $\lambda,\kappa \in \sigma(A)\cup \{0\}$, $\lambda\ne \kappa$. Hence, $n\le|\sigma(B)|\le n$ and $0\notin \sigma(B)$. Let $\mu_1\in \sigma(B)$ be such that $|\lambda_1-\mu_1|< \delta$. For $\mu \in \sigma(B)\setminus \{\mu_1,\overline{\mu_1}\}$, take $\lambda\in \sigma(A)\setminus \{\lambda_1,\overline{\lambda_1}\}$ with $|\lambda-\mu|<\delta$. From $\Re \lambda_1 \ge \Re \lambda+2\delta$, we obtain
\[
  \Re \mu_1 > \Re \lambda_1 -\delta \ge \Re \lambda + \delta > \Re \mu.
\]
Hence, $B\in \mathcal{N}$.

$(ii)$ ${\cal N}$ is dense: Let $B \in \R^{n \times n}$ and $\eps > 0$. We construct an $A \in {\cal N}$ with $\|A - B\| < \eps$. Let $\lambda_1, \dots, \lambda_k$ and $a_1 \pm i b_1, \dots, a_\ell \pm i b_\ell$ be the real and complex conjugate pairs of eigenvalues of $B$ counted with multiplicities, i.e.\ $k + 2\ell = n$. Let $J_0 = T^{-1} B T$ be the Jordan normal form of $B$ for some $T \in \R^{n \times n}$. For $\delta \geq 0$, define a diagonal matrix by letting 
$$D_\delta = {\mbox{diag}}\big(\delta,\ldots,\delta^k,\delta^{k+1},\delta^{k+1},\ldots,\delta^{k+\ell},\delta^{k+\ell}\big)$$
and set $J_\delta = J_0 + D_\delta$. In particular, the off--diagonal $0$'s and $1$'s of $J_0$ remain unaltered. For $\delta > 0$, please note that $J_\delta$ is not the Jordan normal form of $J_0 + D_\delta$. The real eigenvalues of $J_\delta$ are $\lambda_1 + \delta,\ldots,\lambda_k + \delta^k$ and the complex conjugate pairs of eigenvalues of $J_\delta$ are $a_1 + \delta^{k+1} \pm i b_1, \ldots, a_\ell + \delta^{k+\ell} \pm i b_\ell$. 
Note that $J_\delta \in {\cal N}$ for $\delta > 0$ small enough. Using the fact that the map $\delta \mapsto T J_\delta T^{-1}$ is continuous, we can choose $\delta > 0$ such that $A \coloneqq T J_\delta T^{-1}$ satisfies $A \in {\cal N}$ and $\|A - B\| < \eps$.
\end{proof}

For $A \in {\cal N}$, let $G(A) \in \C^{n \times n}$ denote a matrix which conjugates $A$ into (complex) Jordan normal form $J(A) = G(A)^{-1} A G(A) = \diag(\lambda_1, \dots, \lambda_n) \in \C^{n \times n}$ such that $\lambda_1$ is an eigenvalue with largest real part, i.e.\ $\Re(\lambda_1) = \max (\Re (\sigma(A)))$. Note that the columns of $G(A)$ are formed by the eigenvectors of $A$ and, since $A \in \R^{n \times n}$, the eigenvectors corresponding to a complex conjugate pair of eigenvalues of $A$ are also complex conjugate. Using the fact that the eigenvectors of a matrix $A \in {\cal N}$ with distinct eigenvalues depend continuously on the entries of $A$ (see e.g.\ \cite[Ch.\ 2, \S 5.3, p.\ 110]{Kato2012} or \cite[Theorem 3.1.3, p.\ 45]{Ortega1987}), $G(A)$ depends continuously on $A$. With the abbreviations $G(A) = (g_{ij}^A)_{i,j = 1,\dots,n}$, $H(A) \coloneqq G(A)^{-1} = (h_{ij}^A)_{i,j = 1,\dots,n}$, we show that generically $g_{11}^A h_{11}^A \neq 0$.

\begin{lemma}\label{cor:product_generic_nonzero}
The set 
\begin{equation*}
  {\cal M} \coloneqq
  \{A \in \mathcal{N} \,:\, g_{11}^A h_{11}^A \neq 0 \}
\end{equation*}
is open and dense in $\R^{n \times n}$.
\end{lemma}

\begin{proof}
We show that ${\cal M}$ is open and dense in ${\cal N}$. Together with Lemma \ref{thm:matrix_generic}, it follows that ${\cal M}$ is also open and dense in $\R^{n \times n}$.

$(i)$ ${\cal M}$ is open: For $A \in {\cal N}$ the maps $A \mapsto G(A)$, as well as $A \mapsto H(A) = G(A)^{-1}$, are continuous. As a consequence, also the entries $g_{11}^A$ of $G(A)$ and $h_{11}^A$ of $H(A)$ depend continuously on $A$ and therefore the condition $g_{11}^A h_{11}^A \neq 0$ is open in ${\cal N}$.

$(ii)$ ${\cal M}$ is dense: Let $B \in {\cal N}$ and $\eps > 0$. Note that, by Cramer's rule, $h_{11}^B = \det(G_1(B))/\det(G(B))$ with $G_1(B)$ being the matrix formed by replacing the first column of $G(B)$ by the column vector $(1,0,\dots,0)^T$. For $\Delta \in \C^{n \times n}$, with $\|\Delta\|$ small enough, the matrix $B_\Delta \coloneqq (G(B) + \Delta) J(B) (G(B) + \Delta)^{-1}  \in \C^{n \times n}$ is well-defined and $G(B_\Delta) = G(B) + \Delta$. Note that $B_\Delta \in \R^{n \times n}$, if $\Delta$ is chosen such that two columns of $G(B) + \Delta$ are complex conjugate in case they correspond to a complex conjugate eigenvalue pair of $J(B)$. As $\sigma(B)=\sigma(B_\Delta)$, we obtain $B_\Delta \in \mathcal{N}$. In particular, for every $\delta > 0$ small enough, there exists a $\Delta$ with $\|\Delta\| \leq \delta$ such that $B_\Delta \in \mathcal N$, $g_{11}^{B_\Delta} \neq 0$ and, by the density of invertible matrices, also $\det(G_1(B_\Delta))\ne 0$. Consequently, $g_{11}^{B_\Delta} h_{11}^{B_\Delta} \neq 0$ and, thus, $B_\Delta  \in {\cal M}$. By choosing $\delta > 0$  small enough, we also obtain $\|B - B_\Delta\| < \eps$.
\end{proof}

\subsection{Local passivity and dissipativity}

A useful concept in electrical network theory is the notion of a port consisting of a pair of terminals and the current entering one of the terminals is always required to be equal to the current leaving the other terminal (see e.g.\ \cite{Csurgay1965} and also \cite{Zemanian1970} for an extension to the Hilbert space setting and a related notion of passivity). We formulate Chua's notion of local activity \cite[Definition 1]{Chua05} for a real or complex Hilbert space $H$, $A$ the generator of a $C_0$-semigroup $S$ in $H$, $P$ a projection on $H$ and for a fixed basis of $H$ the base elements in $\im P$ correspond to ports or port variables (see also Section \ref{sec:local-activity}). Consider the following class of differential equations
\begin{equation}\label{eq:linear-H}
  \dot x(t) = A x(t) + P u(t)
\end{equation}
with $u \in C(\R_{\geq 0}, H)$. Then $\R_{\geq 0} \ni t \mapsto x(t) = S(t)x_0 + \int_0^t S(t-s) u(s) \,ds \in H$ is the mild solution of the Cauchy problem \eqref{eq:linear-H}, $x(0) = x_0 \in H$.

\begin{definition}[Local activity, local passivity] The pair $(A,P)$, or equivalently equation \eqref{eq:linear-H}, is called \emph{locally passive} if for all $T>0$ and $u\in C([0,T],H)$ the (mild) solution $x\in C([0,T],H)$ of the initial value problem \eqref{eq:linear-H}, $x(0) = 0$, satisfies
\[
   W_T(u) \coloneqq \Re\int_0^T \langle x(t),Pu(t)\rangle \,dt\geq 0.
\]
Equation \eqref{eq:linear-H} is called \emph{locally active} if it is not locally passive, i.e.\ if there exist $T>0$ and $u\in C([0,T],H)$ such that $W_T(u) < 0$.
\end{definition}

\begin{proposition}[Sufficient condition for local passivity]\label{prop:dissipative}
 Let $A$ be the generator of a $C_0$-semigroup in the Hilbert space $H$, $P\in L(H)$ an orthogonal projection. If $\Re \langle PAx,x\rangle\leq 0$  for all $x\in D(A)$ then $(A,P)$ is locally passive.
\end{proposition}
\begin{proof}
Let $T>0$ and $u\in C([0,T],H)$. Denote by $x$ the solution of 
\begin{equation*}
   \dot{x}(t) = Ax(t)+Pu(t), \qquad x(0) = 0.
\end{equation*}
Then, we compute
\begin{align*}
  \Re \int_0^T \langle x(t),Pu(t)\rangle \,dt & = \Re \int_0^T \langle Px(t),\dot{x}(t)\rangle \,dt -\Re \int_0^T \langle Px(t),Ax(t)\rangle \,dt \\ 
                                                    &\geq \frac{1}{2}|Px(T)|^2\geq 0.\qedhere
\end{align*}
\end{proof}

If $P=I$ also the reverse implication of Proposition \ref{prop:dissipative} holds.

\begin{theorem}[Characterization of local passivity for trivial projection]\label{thm:passivity}
  Let $A$ be the generator of a $C_0$-semigroup in the Hilbert space $H$. Then the following statements are equivalent.
\begin{enumerate}[(i)]
 \item $(A,I)$ is locally passive,
 \item $A$ is dissipative, i.e. $\Re\langle Ax,x\rangle \leq 0$ for all $x\in D(A)$.
\end{enumerate}
\end{theorem}
\begin{proof}
$(ii) \Rightarrow (i)$. This was already shown in Proposition \ref{prop:dissipative}. 

$(i) \Rightarrow (ii)$. We denote by $S$ the semigroup generated by $A$. Let $u\in C_c^\infty(\mathbb{R}_{>0},H)$, $b\coloneqq \sup \spt u$ and $\rho>\omega$, where $\omega \in \R$ is such that $\|S(t)\|\leq Me^{\omega t}$ for every $t\geq 0$ and some $M\geq1$. We set
\begin{equation*}
 x(t)=\int_0^t S(t-s)e^{-\rho(t-s)} u(s) \,ds \quad(t\geq0),
\end{equation*}
i.e., $x$ is the unique (classical) solution of 
\begin{equation}\label{eq:Cauchy_prob}
    \dot{x}(t)=Ax(t)+u(t)-\rho x(t) 
    , \quad
    x(0)=0.
 \end{equation}
The local passivity of $(A,I)$ yields $\Re \intop_0^T \langle x(t), u(t)-\rho x(t)\rangle \mathrm{ d}t\geq 0$ for each $T>0$ and hence,
\begin{align*} 
   \rho \int_{0}^T |x(t)|^2 \mathrm{ d}t &\leq \Re\int_{0}^T\langle x(t),u(t)\rangle \mathrm{ d}t\\
                           &=\Re \int_{0}^T \langle x(t),\dot{x}(t)+ \rho x(t)-Ax(t)\rangle \mathrm{ d}t\\
                           & = \frac{1}{2}|x(T)|^2+\rho\int_{0}^T|x(t)|^2\mathrm{ d}t-\Re\int_{T_0}^T\langle Ax(t),x(t)\rangle \mathrm{ d}t.
\end{align*}
Therefore,
\[
   \Re\int_{0}^T\langle Ax(t),x(t)\rangle \mathrm{ d}t\leq \frac{1}{2}|x(T)|^2
   \]
   for every $T>0$ and since
\[
 |x(T)|\leq \intop_0^T M e^{(\omega-\rho)(T-s)} |u(s)|\mathrm{ d}s \leq Me^{(\omega-\rho)T} \intop_0^{b} e^{(\rho-\omega)s} \mathrm{ d}s \|u\|_\infty \to 0\quad (T\to \infty),
\]
we infer that
\begin{equation*}
   \Re\intop_0^\infty\langle Ax(t),x(t)\rangle\mathrm{ d}t \leq 0.
\end{equation*}
Let now $\phi\in C_c^\infty(\mathbb{R}_{>0}), \phi\ne 0$ and $x_0\in D(A)$. Then, $x\coloneqq \phi x_0$ solves (\ref{eq:Cauchy_prob}) for $u\coloneqq \phi'x_0+\rho\phi x_0-\phi Ax_0\in C_c^\infty(\mathbb{R}_{>0},H)$ and hence, 
\begin{equation*}
 0\geq \Re \intop_0^\infty \langle Ax(t),x(t)\rangle \mathrm{ d}t= \Re \langle Ax_0,x_0\rangle |\phi|_{L_2(\mathbb{R}_{\geq0})}^2. 
\end{equation*}
In consequence, we arrive at 
\[
 \Re \langle Ax_0,x_0\rangle\leq 0\quad (x_0\in D(A)).\qedhere
\]
\end{proof}

\section*{Acknowledgements}
B.G.~would like to acknowledge the hospitality of the Center for Dynamics and the Institute for Analysis at the Department of Mathematics, TU Dresden. S.S.~was in part supported by the German Research Foundation (DFG) within the Cluster of Excellence (EXC 1056) ``Center for Advancing Electronics Dresden'' (cfaed). M.W.~carried out this work with financial support of the EPSRC grant EP/L018802/2. This is gratefully acknowledged. We thank Leon Chua for helpful comments.

\end{document}